\documentclass[12pt]{amsart} 
\usepackage{savesym}
\usepackage{amsmath,amsthm,amssymb,mathrsfs,graphicx,scalerel}
\usepackage{newtxmath}
\savesymbol{bigtimes}
\usepackage{mathabx}
\restoresymbol{NEWbigtimes}{bigtimes}
\usepackage{tasks}
\usepackage{extarrows}
\usepackage{centernot}
\usepackage{tikz}
\usetikzlibrary{matrix,graphs,arrows,decorations.pathmorphing,decorations.markings,quotes}
\usepackage{tikz-cd}
\usepackage{float}
\usepackage[toc,page]{appendix}
\usepackage{faktor}
\usepackage{subcaption}
\usepackage{blkarray}
\usepackage{array}
\usepackage{pifont}
\usepackage{mathtools}
\allowdisplaybreaks
\usepackage{mathabx}
\usepackage{verbatim}
\usepackage{url}
\usepackage{color}

\usepackage{pgfplots}
\usepackage{graphicx}
\usepackage{calc}

\newcolumntype{P}[1]{>{\centering\arraybackslash}p{#1}}

\numberwithin{equation}{section}

\newtheorem{theorem}[equation]{Theorem}
\newtheorem{lemma}[equation]{Lemma}
\newtheorem{conjecture}[equation]{Conjecture}
\newtheorem{proposition}[equation]{Proposition}

\newtheorem{corollary}[equation]{Corollary}

\newtheorem*{theorem*}{Theorem}

\theoremstyle{definition}
\newtheorem{definition}[equation]{Definition}
\newtheorem{example}[equation]{Example}

\theoremstyle{remark}
\newtheorem{remark}[equation]{Remark}

\DeclareMathOperator*{\foo}{\scalerel*{\boxplus}{\sum}} 

\tikzset{degil/.style={
            decoration={markings,
            mark= at position 0.5 with {
                  \node[transform shape] (tempnode) {$\backslash$};}},postaction={decorate}}} 

\renewcommand{\H}{\mathbb{H}}
\renewcommand{\P}{\mathbb{P}}
\newcommand{\K}{\mathbb{K}}
\renewcommand{\S}{\mathbb{S}}
\newcommand{\T}{\mathbb{T}}
\newcommand{\R}{\mathbb{R}}
\newcommand{\C}{\mathbb{C}}

\newcommand{\Z}{\mathbb{Z}}

\newcommand{\0}{{\vmathbb 0}}
\newcommand{\1}{{\vmathbb 1}}

\begin{document}
\title{Generalising Kapranov's Theorem For Tropical Geometry over Hyperfields}
\author{James Maxwell}
\address{
Department of Mathematics, Computational Foundry, Bay Campus, Swansea University, Fabian Way, Swansea, SA1 8EN}
\date{April 2022}
\email{825537@swansea.ac.uk, james.william.maxwell@gmail}

\begin{abstract}
Kapranov's theorem is a foundational result in tropical geometry. It states that the set of tropicalisations of points on a hypersurface coincides precisely with the tropical variety of the tropicalisation of the defining polynomial. The aim of this paper is to generalise Kapranov's theorem, replacing the role of a valuation, $\mathit{trop}:K \rightarrow \R\cup\{-\infty\}$, with a more general class of hyperfield homomorphisms, $\H \rightarrow \T$, which satisfy a relative algebraic closure condition. The map, $\eta : \T\C \rightarrow \T$, $\eta(x) : = \mathrm{log}(|x|)$, where $\T\C$ is the tropical complex hyperfield, provides an example of such a hyperfield homomorphism.
\end{abstract}

\maketitle


\thispagestyle{empty}
\section{Introduction}

\noindent Hyperfields are generalisations of fields, where the addition is allowed to be a multivalued operation. The notion of a multivalued operation was introduced by Marty in the mid 1930's, who discussed hypergroups. Then, Krasner in \cite{KRA} progressed the notion to hyperrings, and more recently the connection to tropical geometry has been discussed by Viro in \cite{VIRa}, then by Connes and Consani in \cite{CC}. One of the main objects studied in this paper is $\T\C$, the tropical complex hyperfield. It was first introduced by Viro in \cite{VIR}, and is described as a dequantisation of the complex numbers. The motivation of  Viro's work was to utilise structures with multivalued addition as an algebraic foundation for tropical geometry. There has been substantial progress made in the last several years in the development of the algebraic theory of hyperfields. Baker and Bowler developed a theory of matroids over hyperfields in \cite{BB}, and there has been work done by Bowler and Lorscheid on roots and multiplicities, especially characterising multiplicities for the Krasner, sign and tropical hyperfields, in \cite{BL}. The work completed by Jun, in \cite{JJ} and \cite{JJb}, is a more recent study of algebraic geometry over hyperfields. In \cite{JJb} algebraic sets over hyperfields are introduced and connected to tropical varieties, along with a scheme-theoretic point of view leading to a demonstration that hyperrings without zero divisors can be realised as the hyperring of global regular functions in \cite[Theorem D]{JJb}. One consequence of the multivalued addition is a necessary extension of the notion of a root. Over hyperfields roots are defined as elements at which the polynomial outputs a set that includes zero, rather than exactly equals zero.  \\

\noindent The motivation for this paper comes from Kapranov's theorem in tropical geometry, \cite[Theorem 3.1.3]{MS}, which states, under mild hypotheses, the set of roots of the tropicalisation of a polynomial $p$ coincides with the set of tropicalisations of roots of $p$. Concretely, every root of $\mathit{trop}(p)$ lifts to a root of $p$. The connection is that the tropicalisation map, or valuation $\mathit{trop}:K \rightarrow \R \cup \{-\infty\}$, is naturally a hyperfield homomorphism. \\

\noindent We define a property of maps between hyperfields called \emph{relatively algebraically closed} (RAC). We describe this precisely in Definition \ref{RACdef}, but intuitively the RAC property states that when a univariate polynomial is pushed forward through a hyperfield homomorphism, the roots of the polynomial can be lifted back to roots of the original polynomial. The existence of a hyperfield homomorphism that has this property is demonstrated by the map $\eta =\log(| \cdot |) : \T\C \rightarrow \T$, in Theorem \ref{TCRAC}, where $\T\C$ is the tropical complex hyperfield and $\T:=\R \cup \{-\infty\}$ is the tropical hyperfield. In particular,  the operations on $\T$ are the hyperaddition $\boxplus$, defined as $a \boxplus b = \mathrm{max}(a,b)$, unless $a=b$ then $a \boxplus b = [-\infty,a]$. The multiplication is $\odot:= +$. This is a hyperfield analogue of the idempotent semiring structure that has been studied in \cite{ACGNQ} and more recently presented in both \cite{MS} and \cite{VIRa}.\\

\noindent The main result of this paper is a hyperfield version of Kapranov's theorem for RAC hyperfield homomorphisms which map to $\T$.

\begin{theorem}
Take a RAC hyperfield homomorphism, $f: \H \rightarrow \T$, with induced map of polynomials $f_*: \H[X_1 \, , \dots , \, X_n] \rightarrow \T[X_1 \, , \dots , \, X_n]$, given by applying $f$ coefficient-wise. Then for all polynomials $p \in \H[X_1 \, , \dots , \, X_n]$,
$$ V(f_*(p)) = f(V(p)).$$
\end{theorem}

\noindent A natural question regarding the RAC property of hyperfield homomorphisms is: can we give tractable sufficient conditions to guarantee that a map between hyperfields is RAC? This question is explored in Section \ref{characterisingRAC} and partially answered, stating sufficient conditions for a specific class of hyperfield homomorphisms to be RAC. 

\subsection*{Structure of the paper}
In Section \ref{background} hyperfields and hyperfield homomorphisms will be introduced and the definitions of polynomials and their roots over hyperfields are recalled. This will include examples of common hyperfields. The relative algebraic closure property is defined in Section \ref{background}, along with examples, in particular it is shown that $\eta =\log(|\cdot|):\T\C \rightarrow \T$ has this property. Then in Section \ref{KapsThm} the RAC property is used to construct a proof of a hyperfield version of Kapranov's theorem for maps $\H \rightarrow \T$. Finally, in Section \ref{characterisingRAC} the characterisation of RAC maps is discussed. In particular, the multiplicity bound and inheritance properties are introduced and utilised to provide sufficient conditions for a hyperfield homomorphism to be RAC.

\subsection{Acknowledgements}
The author thanks Jeffrey Giansiracusa for the support and input during the supervision sessions, where the ideas and content for the work was developed. The author would also like to thank Ben Smith, Jaiung Jun and the two anonymous referees for helpful feedback on earlier drafts. This work was carried out during the author's PhD and was supported by Swansea University, College of Science Bursary EPSRC DTP (EP/R51312X/1), and the author would like to show thanks for making the research possible.

\section{Background}\label{background}

\noindent A hyperfield is a generalisation of a field, where the addition operation is allowed to be a multivalued operation, introduced by Marty and developed into hyperrings by  Krasner \cite{KRA}. Hyperfields are one class of generalisations of rings and fields; others include tracts (see \cite{BBa})  and fuzzy rings (see \cite{DWa}). There is a discussion of this range of algebraic structures in \cite{BLa}. The connection between fuzzy rings and hyperrings is described in \cite{GJL}, where a fully faithful functor from hyperfields to fuzzy rings with weak morphisms is constructed. One particular motivation for working with hyperfields is that the tropical semiring, where a foundational theory has been thoroughly developed, with \cite{MS} as a general reference, has a natural analogue as the tropical hyperfield. This section will recall the definitions and notation of hyperfields.

\subsection{Hyperfields}

\noindent Given a set $\H$, then a map $\H \times \H \to P(\H)^*$ is a hyperoperation of $\H$ which is denoted by $\boxplus$. Where $P(\H)^*$ is the power set of $\H$, explicitly, $P(\H)^*$ is the set of all nonempty subsets of $\H$. Then for subsets $A,B \subseteq \H$,
$$ A \boxplus B := \bigcup_{a \in A\, , \,b \in B} (a \boxplus b).$$ 
This definition can be extended for a string of elements. Let $x_1 , ... \, ,x_k \in \H$ then we define their sum as follows;
$$ x_1 \boxplus x_2 \boxplus ... \, \boxplus x_k = \bigcup_{x' \in x_2 \boxplus ... \boxplus x_k} x_1 \boxplus x' \qquad.$$
The hyperoperation $\boxplus$ is called commutative and associative if it satisfies 
\begin{equation}
x \boxplus y = y \boxplus x  \label{hypass}
\end{equation}
and
\begin{equation}
( x \boxplus y) \boxplus z = x\boxplus ( y \boxplus z) \label{hypcom}
\end{equation}
respectively. 
\begin{definition} 
A canonical hypergroup is a tuple $(\H, \boxplus ,\0)$, where $\boxplus$ is a commutative and associative hyperoperation on $\H$ such that:
\begin{itemize}
\item (H0) \, $\0 \boxplus x = \{ x \}, \quad \forall \, x \in \H$.
\item (H1) \,  For every $x \in \H$ there is a unique element of $\H$, denoted $-x$, such that $\0 \in x \boxplus -x$.
\item (H2) \, $x \in y \boxplus z$ iff $z \in x \boxplus (-y)$. This is normally referred to as $\mathit{reversibility}$.
\end{itemize}
The reversibility condition is not required for non-canonical hypergroups, but throughout this work only canonical hypergroups will be used so the label canonical is dropped.
\end{definition}

\begin{definition}
A hyperring is a tuple $(\H, \odot ,\boxplus , \1 , \0)$ such that:

\begin{itemize}
\item $(\H , \odot , \1)$ is a commutative monoid.
\item $(\H , \boxplus , \0)$ is a commutative hypergroup.
\item (Absorption rule) $\0 \odot x = x \odot \0 = \0$ \, for all \, $x \in \H$.
\item (Distributive law) \, $a \odot ( x \boxplus y) = ( a \odot x) \boxplus (a \odot y)$ \, for all $a,x,y \in \H$.
\end{itemize}
\end{definition}

\begin{definition}
A hyperring $\H$ is called a hyperfield if $ \0 \neq \1$ and every non-zero element of $\H$ has a multiplicative inverse.
\end{definition}

\begin{remark}\label{context}
From this point onward, to clarify context, when discussing results over a field $K$ the following notation will be used: $+,\times (\text{or}\, \cdot)$ and $\sum $. Whereas, when discussing results over a hyperfield $\H$ the following notation will be used: $ \boxplus \, , \, \odot$ and $\foo $. 
\end{remark}

\subsection{Examples of Hyperfields}
The following are some key examples of common hyperfields used in the literature. They will be used to demonstrate the properties that are defined in this paper.

\begin{example}\label{Ex:trop-hyp}
The tropical hyperfield has the underlying set $\R \cup \{-\infty\}$ and it is usually denoted as $\T$. The multiplication on $\T$ is an extension of the addition on $\R$:
$$ x \odot y := x+y \quad \text{and} \quad x \odot -\infty = x + -\infty = -\infty.$$
The hyperaddition is defined as the following multivalued operation:
\begin{equation}
x \boxplus y = 
\begin{cases}
\{ \mathrm{max}(x,y)\}, \quad & \text{if} \quad x \neq y\\
\{ z \, | \, z \le x\} \cup \{-\infty\}, \quad & \text{if} \quad x = y
\end{cases}
\end{equation}
The additive neutral element $\0$ of $\T$ is $-\infty$ and the multiplicative neutral element $\1$ of $\T$ is 0. The tropical hyperfield is a hyperfield analogue of the tropical semiring described in \cite{MS}. The tropical hyperfield can also be defined using \emph{min} instead of \emph{max}, and $\{\infty\}$ instead of $\{-\infty\}$, which yields an isomorphic structure. For further discussions of the tropical hyperfield and demonstrations of its usefulness, see both \cite{VIR} and \cite{VIRa}.
\end{example}

\noindent The next example is a key part of the results outlined in the following sections. The main purpose it has is as a candidate to replace the valued field in Kapranov's theorem. 

\begin{example}
The tropical complex hyperfield is usually denoted $\T\C$ and has the complex numbers $\C$ as its underlying set. The standard complex multiplication is given to $\T\C$. The hyperaddition is defined in the following way for all $z,w \in \C$:
\begin{equation} 
z \boxplus w = \nonumber
\begin{cases}
\{ c \in \C : |c| \le z \}, & \quad \text{if} \quad  w = -z.\\
z, & \quad \text{if}\quad |z| > |w|.\\
w, & \quad \text{if} \quad |w| >|z|.\\
\text{Shortest arc connecting}\; z \; \text{and} \; w,\; \text{with radius}\; |z|, & \quad \text{if} \, |z|=|w|, \, z\neq \pm w. 
\end{cases}
\end{equation}
\begin{figure}
\begin{center}
    \includegraphics{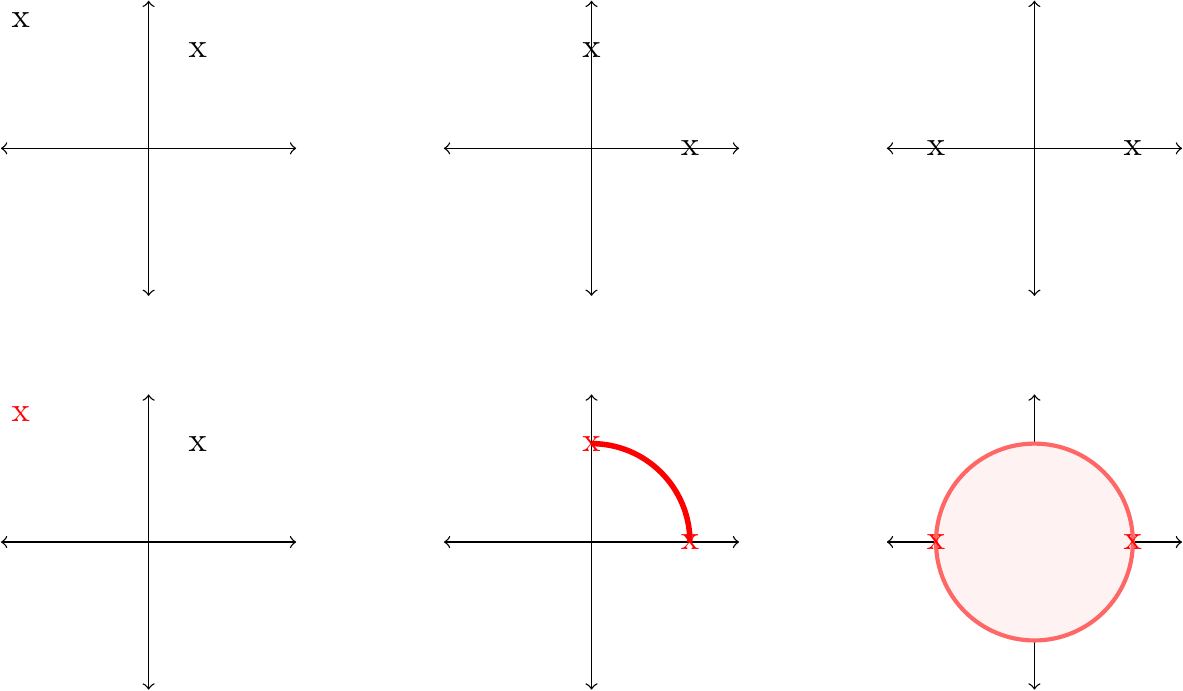}
    \caption{hyperaddition in $\T\C$ - Where the first row is two points in $\T\C$ and the bottom row represents the outcome of the addition in red.}
    \label{HypAddTC}
    \end{center}
\end{figure}
With this hyperaddition $\T\C$ is a hyperfield. The additive neutral element $\0$ of $\T\C$ is $0$ and the multiplicative neutral element $\1$ of $\T\C$ is $1$. This hyperfield is discussed in \cite[example 9]{AD} and was introduced in \cite[ection 6]{VIR} where it is described as the dequantization of the field of complex numbers. For further intuition as to the behaviour of the hyperaddition of $\T\C$ see Figure 1 and \cite[figure 1 in section 6]{VIR}.
\end{example}

\begin{example}
A field $K$ can be viewed as a hyperfield in a trivial manner, where the hyperaddition is defined as $x \boxplus y = \{x+y\}$. 
\end{example}

\begin{example}
The Krasner hyperfield has the underlying set $\{0,1\}$, and it is usually denoted $\K$. The multiplication for $\K$ is defined as:
$$0 \odot 0 = 0, \, 1 \odot 1 = 1, \, 0 \odot 1 = 0. $$
The hyperaddition for $\K$ is defined as:
$$ 0 \boxplus x = x \boxplus 0 = \{ x \} \quad \text{for} \quad x = 0 , 1$$ 
$$ 1 \boxplus 1 = \{ 0 , 1\}.$$
The additive neutral element $\0$ of $\K$ is 0 and the multiplicative neutral element $\1$ of $\K$ is 1.
\end{example}

\begin{example}
The hyperfield of signs has the underlying set $\{-1,0,1\}$, and it is usually denoted by $\S$. The multiplication is the restriction of the usual multiplication on the real numbers and the hyperaddition defined as:
$$ 0 \boxplus 0 = 0, \quad 0 \boxplus 1 = 1, \quad 0 \boxplus -1 = -1,$$
$$ 1 \boxplus 1 = 1, \quad -1 \boxplus -1 = -1,$$
$$ 1 \boxplus -1 = \{0 , 1 , -1\}.$$
The additive neutral element $\0$ of $\S$ is 0 and the multiplicative neutral element $\1$ of $\S$ is 1.
\end{example}

\begin{example}
The phase hyperfield, usually denoted by $\P$, has the underlying set of $S^1 \cup \{0\}$, where $S^1 = \{ e^{i\theta} \in \C\, | 0 \le \theta < 2\pi \, \}$ , which is the complex unit circle union with zero. Multiplication on $\P$ is inherited from $\C$, and the hyperaddition is defined by the rule,
$$\text{if} \quad \theta_1 = \theta_2 + \pi, \quad \text{then} \quad e^{i\theta_1} \boxplus e^{i\theta_2} = \{ 0 ,e^{i\theta_1} , e^{i\theta_2} \}.$$
$$\text{If} \quad \theta _1 < \theta_2 < \theta_1 + \pi, \quad \text{then} \quad e^{i\theta_1} \boxplus e^{i\theta_2} = \{e^{i\theta} \, | \, \theta_1 < \theta < \theta_2 \}.$$
$$\text{If} \quad \theta _2 < \theta_1 < \theta_2 + \pi, \quad \text{then} \quad e^{i\theta_1} \boxplus e^{i\theta_2} = \{e^{i\theta} \, | \, \theta_2 < \theta < \theta_1 \}.$$
The additive neutral element $\0$ of $\P$ is 0 and the multiplicative neutral element $\1$ of $\P$ is 1. See \cite{BB} and \cite{BS} for alternative but equivalent definitions of the phase hyperfield. 
\end{example}

\subsection{Hyperfield Homomorphisms}

\begin{definition}
Given hypergroups $\H_1$ and $\H_2$, with respective hyperoperations $\boxplus_1$ and $\boxplus_2$, a hypergroup homomorphism is a map $f : \H_1 \to \H_2$, such that $f(\0)=\0$ and $f(x \boxplus_1 y) \subseteq f(x) \boxplus_2 f(y)$ for all $x,y \in \H_1$.
\end{definition}

\begin{definition}
Given hyperrings $\H_1$ and $\H_2$, with respective hyperoperations $\boxplus_1$ and $\boxplus_2$ and multiplication $ \odot_1$ and $\odot_2$, a homomorphism of hyperrings $f : \H_1 \to \H_2$ is a map such that:
\begin{enumerate}
\item $f(x \boxplus_1 y) \subseteq f(x) \boxplus_2 f(y)$ \quad \text{and} \quad $f(\0)=\0.$
\item $f(x \odot_1 y) = f(x) \odot_2 f(y)$ \quad \text{and} \quad $f(\1)=\1.$
\end{enumerate}
I.e., this is a homomorphism of additive hypergroups and a homomorphism of multiplicative monoids.
\end{definition} 

\begin{definition}
A map $f: \H_1 \rightarrow \H_2$ between hyperfields is a hyperfield homomorphism if it is a hyperring homomorphism. 
\end{definition}

\begin{example}\label{Ex:hyp-homo}
The following are common examples of hyperfield homomorphisms.
\begin{equation} \nonumber 
    f:\H \rightarrow \K, \quad f(x) =
    \begin{cases}
        1, & \text{if} \quad x\neq 0. \\
        0, &\text{if} \quad x=0.
    \end{cases}
\end{equation}
\begin{equation} \nonumber
    \mathit{ph}:\C \rightarrow \P, \quad \mathit{ph}(x) =
    \begin{cases}
        \frac{x}{|x|}, & \text{if} \quad x \in \C\backslash\{0\}.\\
        0, &\text{if} \quad x=0.
    \end{cases}
\end{equation}
\begin{equation} \nonumber
 \mathit{sgn}: \R \rightarrow \S, \quad \mathit{sgn}(x) = 
 \begin{cases}
 1, & \text{if} \quad  x \in \R_{>0}. \\
 -1, & \text{if} \quad x \in \R_{<0}. \\
 0, & \text{if} \quad x=0.
 \end{cases}
 \end{equation}
A discussion of these and other hypefield homomorphims can be found in \cite{LAJD}.
\end{example}

\noindent The next example will be a description of the hyperfield homomorphism  which is the basis for the results in this paper. 

\begin{example}\label{eta}
There exists a hyperfield homomorphism from the tropical complex hyperfield to the tropical hyperfield. The map is denoted $ \eta: \T\C \rightarrow \T$, and is defined as:
$$ \eta(z) : = \mathrm{log}(|z|).$$
This homomorphism has previously appeared in a slightly different guise in \cite{LAJD}. Firstly, in \cite[page 341]{LAJD} the tropical triangle hyperfield, $\T\triangle := \R_{\ge 0}$ is defined and the hyperfield homomorphism $|\cdot|: \T\C \rightarrow \T\triangle$ is given as the standard absolute value. This map can also be seen as the lower right arrow of \cite[figure 1]{LAJD}. When $\T\triangle$ is discussed in \cite{LAJD} it is stated that logarithm map, $\mathrm{log}: \T\triangle \rightarrow \R \cup \{-\infty\}$, induces a hyperfield structure on $\R \cup \{-\infty\}$. This hyperfield structure is precisely the tropical hyperfield, denoted in this work as $\T$, which is defined in Example \ref{Ex:trop-hyp}. The motivation for working with the composition map $ \eta(z) : = \mathrm{log}(|z|)$ is that the tropical hyperfield is the multivalued analogue of the tropical semiring, and therefore has been studied in more detail. The tropical triangle hyperfield is isomorphic to the tropical hyperfield under the logarithm map, but has very few mentions in the literature. Moreover, because of the similarities between the tropical semiring and the tropical hyperfield, it is more natural to be discussing a generalised version of Kapranov's theorem utilising $\T$ rather than $\T\triangle$. 
\end{example}

\subsection{Polynomials over Hyperfields}

This section will describe how polynomials and their roots are defined over hyperfields.

\begin{definition}
The set of polynomials in $n$-variables over a hyperfield $\H$ will be denoted $\H[X_1 \, , \dots , \, X_n]$, where elements of this set are defined as: 
\begin{equation}\label{hyperpolynomial}
  \foo_I c_I \odot X^{i_1} \odot \, \cdots \, \odot X^{ i_n} = \foo_I c_I \odot \underline{X}^I,
\end{equation}
where multi-index notation is used and $I = (i_1 \, , \dots , \, i_n) \subset \Z^n$ and $c_I \in \H$.
\end{definition}

\begin{remark}
Note that the notation $\H[X_1 \, , \dots , \, X_n]$ is used only to denote the \emph{set} of polynomials over $\H$. In general there is no ring or hyperring structure on $\H[X_1 \, , \dots , \, X_n]$, unlike the specialised case where $\H$ is a field. Even in the univariate case it is easy to see that the set of polynomials is not a hyperring; the multivalued nature of the addition in $\H$ combined with the distributivity leads to products of polynomials also being multivalued. For example, $(aX \boxplus b)\odot(cX \boxplus d) \subseteq acX^2 \boxplus (ad \boxplus bc)X \boxplus bd.$ The coefficient $(ad \boxplus bc)$ is not necessarily single valued, which shows that multiplication of polynomials is multivalued, hence $\H[X_1 \, , \dots , \, X_n]$ is not a hyperring. See \cite[remark 4.4 and example 4.13]{JJ} for an explicit example where the behaviour of $\H[X_1 \, , \dots , \, X_n]$ is more controllable.
\end{remark}

One difference to note between polynomials defined over a field and those defined over a hyperfield occurs when the polynomial is evaluated at an element of $\H$. When a polynomial defined over a hyperfield is evaluated at an element, the output, unlike over a field, can be a set of elements. This leads to the following notion of a root of a polynomial over $\H$. 

\begin{definition}
Let $p = \foo_I c_I \odot \underline{X}^I \in \H[\underline{X}]$. An element $\underline{a} = (a_1 \, , \dots , \, a_n)$ is a \emph{root} of the polynomial if $\0 \in p(\underline{a}) =  \foo_I c_I \odot \underline{a}^I$.\\
\noindent This allows for a natural definition of the variety of $p(X_1 \, , \dots , \, X_n)$ as,
$$ V(p) : = \{ \underline{a} = (a_1 \, , \dots , \, a_n) \in \H^n \, | \, \0 \in p(\underline{a})\}. $$
\end{definition}

\noindent The next definition recalls the notion of the multiplicity of a root for univariate polynomials defined over hyperfields.  

\begin{definition}(\cite[Def. 1.5]{BL})
Let $p(X) \in \H[X]$, the multiplicity of an element $a \in \H$ is denoted $\mathit{mult}_a(p)$ and defined as,
\begin{equation}\label{multidef}
\mathit{mult}_a(p) = 
\begin{cases}
0, \quad \text{if} \, a \in \H \, \text{is not a root.}\\ 
1 + \mathrm{max}\{ \mathit{mult}_a(q), \, \text{if} \quad \exists \, q \in \H[X],\, \text{s.t.} \, p \in (X \boxplus -a) \odot q(X)\}.
\end{cases}
\end{equation}
\end{definition}

\noindent Because of the multivalued nature of multiplication of polynomials with coefficients in a hyperfield, the polynomial $q(X)$ in $\eqref{multidef}$ is not necessarily unique. This is the motivation behind the recursive definition of the multiplicity. See \cite{BL} for details on the original definition and examples of the non-uniqueness. 

\begin{example}
\begin{enumerate}
\item If $p(X) = X^2 \boxplus X \boxplus 1 \in \K[X]$, then $1 \in V(p)$, as $$p(1) = 1 \boxplus 1 \boxplus 1 =\K \ni \0.$$
\item If $p(X) = X^2 \boxplus X \boxplus -1 \in \S[X]$, then $-1 \in V(p)$, as $$p(-1) = 1 \boxplus -1 \boxplus -1 = \S \ni \0.$$
\end{enumerate}
\end{example}

\begin{remark}
There is a complete description of roots and corresponding multiplicities for univariate polynomials over $\K$, $\S$ and $\T$ in \cite{BL}, where the results are used to demonstrate proofs of Descartes' Rule of Signs and Newton's Polygon Rule. 
\end{remark}

\noindent Now that polynomials have been defined, their structure can be combined with hyperfield homomorphisms to describe an induced map of polynomials over hyperfields. 

\begin{definition}
Let $f:\H_1 \rightarrow \H_2$ be a hyperfield homomorphism. This induces a map from polynomials with coefficients in $\H_1$ to polynomials with coefficients in $\H_2$. This map is denoted $f_* : \H_1[X_1 \, , \dots , \, X_n] \rightarrow \H_2[X_1 \, , \dots , \, X_n]$, and is defined for $p = \foo_I c_I \odot \underline{X}^I \in \H_1[X_1 \, , \dots , \, X_n]$ as;
$$ f_*(p)  = \foo_I f(c_I) \odot \underline{X}^I \in \H_2[X_1 \, , \dots , \, X_n].$$
(Note: the hyperoperations are now the operations over $\H_2$, and $f_*(p)$ will be called the push-forward of $p$.)
\end{definition}

\begin{example}
Take the polynomial $p = 4X^2 -5X +1 \in \R[X]$, the hyperfield homomorphism $\mathit{sgn}:\R\rightarrow\S$ induces the map $\mathit{sgn}_*: \R[X] \rightarrow \S[X]$, which gives $\mathit{sgn}_*(p) = X^2 \boxplus -X \boxplus 1 \in \S[X]$.
\end{example}

\noindent Note that roots of $p \in \H_1[X]$ push-forward to roots of $f_*(p) \in \H_2[X]$. This is precisely described in Lemma \ref{pushforwardgeneral}. Next the definition of a \emph{relatively algebraically closed} hyperfield homomorphism is introduced.

\begin{definition}\label{RACdef}
Let $f:\H_1 \rightarrow \H_2$ be a hyperfield homomorphism, with induced map $f_*: \H_1[X] \rightarrow \H_2[X]$. We say that $f$ is \emph{relatively algebraically closed} (RAC) if for all univariate polynomials  $p \in \H_1[X]$ and every root $b \in V(f_*(p))$, there exists $a \in f^{-1}(b)$ such that $a \in V(p)$.
\end{definition}

\noindent The map $f:\H_1 \rightarrow \H_2$ being RAC has the  immediate consequence that $V(f_*(p)) \subseteq f(V(p))$ for all  $p \in \H_1[X]$. Then, combing with Lemma \ref{pushforwardgeneral}, $V(f_*(p))= f(V(p))$.

\begin{example}
    The RAC property is defined for all univariate polynomials over $\H_1$, and does not require the map $f:\H_1 \rightarrow \H_2$ to be surjective. To demonstrate, consider a transcendental field extension $\mathbb{Q} \rightarrow \mathbb{Q}(\alpha)$. The inclusion map $i:\mathbb{Q} \hookrightarrow \mathbb{Q}(\alpha)$ is RAC but clearly not surjective. 
\end{example}

\begin{definition}
A hyperfield $\H$ is called \emph{algebraically closed} if every univariate polynomial has a root in $\H$. 
\end{definition}

The Krasner hyperfield $\K$ is algebraically closed, by \cite[Remark 1.11]{BL}. Thus, a hyperfield $\H$ is algebraically closed if and only if the canonical map (from Example \ref{Ex:hyp-homo}) $f:\H \rightarrow \K$ is RAC. \\

The following example outlines a RAC map which is the main motivating example for this paper. It is the map that is studied in tropical geometry, and is the basis for the generalisation to hyperfield homomorpshisms. 

\begin{example}\label{trop}
Let $K$ be an algebraically closed field with surjective valuation, $\mathit{trop}: K \rightarrow \R \cup\{-\infty\}$. This is RAC hyperfield homomorphism.  This is the underlying structure investigated when discussing valuations and tropicalisation maps in relation to tropical geometry. See \cite[section 2.1 and theorem 3.1.3]{MS} for a more detailed description. 
\end{example}

\begin{example}
Take the map $\mathit{sgn}: \R \rightarrow \S$, and $p = X^2 -X +1 \in \R[X]$. Then, $\mathit{sgn}_*(p) = X^2 \boxplus -X \boxplus 1 \in \S[X]$. The polynomial $p$ has an empty variety, whereas $\mathit{sgn}_*(p)(1) = 1 \boxplus -1 \boxplus 1 = \S$, so $1 \in V(\mathit{sgn}_*(p))$.
This demonstrates that the map $\R \rightarrow \S$ is not a RAC map. 
\end{example}

\begin{example}
Take the map $\mathit{ph}:\C \rightarrow \P$, and the polynomial $p = X^2 + X +1 \in \C[X]$. It is shown in \cite[Remark 1.10]{BL}, that $\mathit{ph}_*(p)$ has a root at each $a =e^{i\theta}$ for all $\pi / 2 < \theta < 3\pi /2$. Not every element of this set can be lifted and hence $\mathit{ph}:\C \rightarrow \P$ is not a RAC map. 
\end{example}

\subsection{Connecting RAC to $\T\C$}

\noindent The remainder of the section will be focused on demonstrating that the map $\eta: \T\C \rightarrow \T$ satisfies the RAC property.

\begin{proposition}\label{TCroots}
Let $p = \foo_{i=0}^n c_i \odot X^i \in \T\C[X]$. An element $a \in \T\C$ is a root of $p$ if there exist $j_1 \, , \dots , \, j_m \in \{1 \, , \dots , \, n\}$ such that 
$$ |c_{j_1}a^{j_1}| = \dots = |c_{j_m}a^{j_m}| > |c_ia^i| \quad \forall i \not \in \{j_1 \, , \dots , \, j_m\} \quad
\text{and} \quad -c_0 \in \foo_{k=1}^m c_{j_k} \odot a^{j_k}.$$
\end{proposition}

\begin{proof}
By the definition of the hyperaddition over $\T\C$ the monomial terms with the largest absolute value contribute to the hypersum when the polynomial $p$ is evaluated at $a$. Thus, by the hypothesis,
\begin{align}
     p(a) & = \foo_{i=0}^n c_i \odot a^i \nonumber \\
    & = \foo_{k=1}^m c_{j_k} \odot a ^ {j_k} \boxplus c_0, \nonumber
\end{align}
so, $\0 \in p(a)$ if and only if $-c_0 \in \foo_{k=1}^m c_{j_k} \odot a^{j_k}$.
\end{proof}

\begin{proposition}\label{prop:trop_roots}
Given $p \in \T[X]$, an element $a \in \T$ is a root of $p$ if and only if there are two or more monomial terms in $p(a)$, such that they are equal and greater than or equal to the remaining monomial terms.
\end{proposition}

\begin{proof}
Evaluating $p$ at $a \in \T$ and utilising the hyperaddition over $\T$ yields:
$$p(a) = \foo_{i=0}^n c_i \odot a^i = \foo_{k=1}^m c_{j_k} \odot a^{j_k},$$
where the monomials terms $c_{j_k} \odot a^{j_k}$, indexed over $1, \dots ,m$ are all equal and greater that all other monomial terms. If $m=1$ then $p(a) = c_{j_k} \odot a^{j_k} \notni \0$, (ignoring the case where $p(a)$ is a monomial and $a = -\infty$). Whereas, if $m \ge 2$, then
$$ p(a) = c_{j_1} \odot a^{j_1} \boxplus \dots \boxplus c_{j_1} \odot a^{j_1} \ni \0$$ by definition. 
\end{proof}

\begin{lemma}\label{pullbackordering}
Given $a,b \in \T$, if  $a>b$, then for all $\alpha \in \eta^{-1}(a)$ and $\beta \in  \eta^{-1}(b)$, it holds that $|\alpha| > |\beta|$.
\end{lemma}
\begin{proof}
Taking $a,b \in \T$ such that $a>b$, if $\alpha \in \eta^{-1}(a)$, and $\beta \in  \eta^{-1}(b)$ then $\mathrm{log}(|\alpha|) = a > b = \mathrm{log}(|\beta|)$. As both the logarithm and exponential functions preserve order, $\mathrm{log}(|\alpha|) >  \mathrm{log}(|\beta|) \Rightarrow |\alpha| > |\beta|$, as required.
\end{proof}

\begin{theorem}\label{TCRAC}
The hyperfield homomorphism $\eta: \T\C \rightarrow \T$ is RAC. 
\end{theorem}
\begin{proof}
A polynomial $p = \foo_{i=0}^n c_i \odot X^i \in \T\C[X]$ has push-forward $q = \eta_*(p) = \foo_{i=0}^n \eta(c_i) \odot X^i \in \T[X]$. By Proposition \ref{prop:trop_roots}, $a \in V(q)$ if and only if there are two or more monomials terms that are equal and great than or equal than the other monomial terms in $q(a)$. For this to occur there must exist $k_1 \, , \dots , \, k_m \in \{1 \, , \dots , \, n\}$ such that, $\eta(c_{k_1}) \odot a^{k_1} = \, \dots \, = \eta(c_{k_m}) \odot a^{k_m} > \eta(c_i) \odot a^i$ for all $i \in \{1\, ,\dots,\, n\}\backslash \{k_1 \, ,\dots, \,k_m\}$. Explicitly, the maximum is achieved more than once in the terms $\eta(c_1) \odot a^1 \, , \dots , \, \eta(c_n) \odot a^n$, and this maximum is greater than or equal to $\eta(c_0)$.

\noindent As $m \ge 2$, this allows for $t,t' \in \{k_1 \, , \dots , \, k_m\}$ to be chosen to construct a lift of $a$ as follows:
$$ \widetilde{a} = \Big(\frac{-c_{t'}}{c_t} \Big)^{\frac{1}{t-t'}}.$$

\noindent To confirm that $\tilde{a}$ is a lift of $a$, apply the map $\eta$ to $\tilde{a}$ to give 
$$\eta(\tilde{a}) = \Big( \eta(-c_{t'} )\odot (\eta(c_t))^{-1}\Big)^{\frac{1}{t-t'}}.$$
Furthermore, observe that $\eta(c_{t}) \odot a^t = \eta(c_{t'}) \odot a^{t'}$. Then, 
\begin{align}
    \eta(c_{t}) \odot a^t = \eta(c_{t'}) \odot a^{t'} & \Rightarrow a^{t-t'} = \eta(c_{t'} )\odot (\eta(c_t))^{-1} \nonumber \\
    & \Rightarrow a = \Big( \eta(c_{t'} )\odot (\eta(c_t))^{-1}\Big)^{\frac{1}{t-t'}}. \nonumber 
\end{align}
As $\eta(x) = \mathrm{log}(|x|) = \mathrm{log}(|-x|) = \eta(-x)$, this shows that $a = \eta(\tilde{a})$ as required. 

\noindent It remains to shown that $\widetilde{a}$ is a root of the original polynomial $p$. Note that due to Lemma \ref{pullbackordering}, 
\begin{equation} \label{eqn:ab_val_pullback_ineq}|c_{k_1}\widetilde{a}^{k_1}| = \, \dots \, = |c_{k_m}\widetilde{a}^{k_m}| > |c_i\widetilde{a}^i|, \quad \text{for all} \quad i \not \in \{k_1 \, , \dots , \, k_m\}.
\end{equation}
Furthermore, observe the relationship between the monomial terms with $t,t'$ exponents,
\begin{align}
 c_t\widetilde{a}^t \boxplus c_{t'}\widetilde{a}^{t'} & =  c_t \Big(\Big(\frac{-c_{t'}}{c_t} \Big)^{\frac{1}{t-t'}} \Big)^t \boxplus c_{t'} \Big(\Big(\frac{-c_{t'}}{c_t} \Big)^{\frac{1}{t-t'}} \Big)^{t'} \nonumber \\
 & = (-1)^{\frac{t}{t-t'}} (c_t)^{\frac{-t'}{t-t'}}(c_{t'})^{\frac{t}{t-t'}} \boxplus (-1)^{\frac{t'}{t-t'}} (c_t)^{\frac{-t'}{t-t'}}(c_{t'})^{\frac{t}{t-t'}}\nonumber \\
 & = (-1)(-1)^{\frac{t'}{t-t'}} (c_t)^{\frac{-t'}{t-t'}}(c_{t'})^{\frac{t}{t-t'}} \boxplus (-1)^{\frac{t'}{t-t'}} (c_t)^{\frac{-t'}{t-t'}}(c_{t'})^{\frac{t}{t-t'}} \nonumber \\
 & = \{ z \in \C \, : \, |z| \le R\}, \quad \text{where} \quad R = |c_t\widetilde{a}^t| \geq |c_0|. \nonumber 
 \end{align}
Then, $\foo_{j=1}^m c_{k_j} \odot \widetilde{a}^{k_j} \ni -c_0$, which combined with \eqref{eqn:ab_val_pullback_ineq} and Proposition \ref{TCroots} gives that $\tilde{a} \in V(p)$.
\end{proof}

\begin{corollary}
The tropical complex hyperfield $\T\C$ is algebraically closed. 
\end{corollary}
\begin{proof}
It is a classical result in tropical geometry that the tropical semiring is algebraically closed, see for example \cite{GW}. This is trivially equivalent to the statement that the hyperfield $\T$ is algebraically closed. As the map $\eta: \T\C \rightarrow \T$ is RAC and surjective, every polynomial $p \in \T\C[X]$ has a corresponding polynomial $\eta_*(p) \in \T[X]$, which has a root and can be lifted back to a root of $p$. Hence, every polynomial in $\T\C[X]$ has a root.
\end{proof}

\begin{remark}
    The results outlined in \cite{BL} regarding root multiplicities over $\T$ and $\K$ demonstrate from a hyperfield perspective why they are both algebraically closed.
\end{remark}

\begin{example}
Take the polynomial $p = i \odot X^2 \boxplus \Big( \frac{-1 +i\sqrt{3}}{2}\Big ) \odot X \boxplus -1 \in \T\C[X]$, then the push-forward is $\eta_*(p) = 0 \odot X^2 \boxplus 0 \odot X \boxplus 0 \in \T[X]$. It can be seen that $0 \in V(\eta_*(p))$. In accordance with the proof of Theorem \ref{TCRAC}, take $t = 2$ and $t' =1$. This gives,
$$ c_t = c_2 = i, \quad- c_{t'} = \textcolor{teal}{-}c_1 = \Big( \frac{1 -i\sqrt{3}}{2}\Big ), \quad 1/(t-t') = 1$$ 
Taking the template for the lift, 
$$ \widetilde{a} = \Big(\frac{-c_{t'}}{c_t} \Big)^{\frac{1}{t-t'}} = \Big( \frac{1-i\sqrt{3}}{2i}\Big). $$
Now to confirm that this is a pull back of $0$ and is a root of $p(X)$:
$$ \Big | \Big( \frac{1-i\sqrt{3}}{2i}\Big) \Big| = \frac{1}{2} |1-i\sqrt{3}| = 1 \Rightarrow f\Big( \frac{1-i\sqrt{3}}{2i}\Big) = 0.$$
Finally, 
\begin{align}
p(\widetilde{a}) = p\Big( \frac{1-i\sqrt{3}}{2i}\Big) & = i \Big( \frac{1-i\sqrt{3}}{2i}\Big)^2 \boxplus \Big( \frac{-1 +i\sqrt{3}}{2}\Big )\Big( \frac{1-i\sqrt{3}}{2i}\Big) \boxplus -1 \nonumber \\
& = \Big( \frac{i-\sqrt{3}}{2}\Big ) \boxplus \Big( \frac{1 +i\sqrt{3}}{2i}\Big ) \boxplus -1 \nonumber \\
& = \Big( \frac{i-\sqrt{3}}{2}\Big ) \boxplus -\Big( \frac{i - \sqrt{3}}{2}\Big ) \boxplus -1 \nonumber \\
& = \{ z \in \C \, : \, |z| \le 1\} \boxplus -1 \ni \0. \nonumber 
\end{align}
This shows that $\widetilde{a}$ is a lifted root, which demonstrates the application of the structure of the proof of Theorem \ref{TCRAC}.
\end{example}

\section{A Generalisation of Kapranov's Theorem}\label{KapsThm}

\noindent This section will present a specific generalisation of Kapranov's Theorem over hyperfields. The theorem by Kapranov \cite[theorem 3.1.3]{MS} is a key result in tropical geometry, which leads to the Fundamental Theorem of tropical geometry in \cite{MS}. Recall the statement of Kapranov's Theorem in tropical geometry. 

\begin{theorem}
Given an algebraically closed field $K$ with surjective valuation denoted $\mathit{trop}: K \rightarrow \R \cup \{-\infty\}$, for a Laurent polynomial $p = \sum_{I \in \Z^n} c_{I}X^{I} \in K[X_1^{\pm 1} \, , \dots , \, X_n^{\pm 1}]$, 
$$ V(\mathit{trop}(p)) = \mathit{trop}(V(p)).$$
(For further details, see \cite[Theorem 3.1.3]{MS}).
\end{theorem}

\noindent An essential point to recognise is that replacing the valuation with an arbitrary hyperfield homomorphism, one containment holds automatically.

\begin{lemma}\label{pushforwardgeneral}
Let $f: \H_1 \rightarrow \H_2$ be a hyperfield homomorphism. For $p = \foo_{I} c_{I} \odot \underline{X}^{I} \in \H_1[X_1 \, , \dots , \, X_n]$, 
$$ f(V(p)) \subseteq V(f_*(p)).$$
\end{lemma}
\begin{proof}
By definition $ f_*(p) = \foo_{I} f(c_{I}) \odot \underline{X}^{I}  \in \H_2[X_1 \, , \dots , \, X_n]$. Let $\underline{a} = (a_1 \, , \dots , \,a_n) \in \H^n_1$ be a root of $p$, meaning $\underline{a} \in V(p)$ and 
\begin{equation}
\0 \in p(\underline{a}) =p(a_1 \, , \dots , \, a_n) = \foo_{I} c_{I} \odot \underline{a}^{I}.
\end{equation} 
The aim is to demonstrate that $f(\underline{a}) \in V(f_*(p))$ holds. Firstly, 
\begin{align}
f_*(p)(f(\underline{a})) & =  \foo_{I} f(c_{I}) \odot f(\underline{a})^{I} \nonumber \\
& = \foo_{I} f(c_{I} \odot \underline{a}^{I}) \nonumber \\
& \supseteq f \Big ( \foo_{I} c_{I} \odot \underline{a}^{I}\Big) \nonumber \\
& \ni f(\0) = \0 \nonumber 
\end{align}
The above steps use the properties of a hyperfield homomorphism to give that $\0 \in f(p)(f(\underline{a}))$, yielding $f(\underline{a}) \in V(f_*(p))$. Hence, $f(V(p)) \subseteq V(f_*(p))$ holds. 
\end{proof}

\noindent The main result of the paper is to specifically generalise Kapranov's Theorem to hyperfield homomorphisms $f:\H \rightarrow \T$, which satisfy the RAC property. The following theorem demonstrates that the RAC property for a hyperfield homomorphism $\H \rightarrow \T$ can be used to deduce the existence of lifts of roots for polynomials in $n$-variables. 

\begin{theorem}[Generalised Kapranov's Theorem]
Given a polynomial $p \in \H[X_1 \, , \dots , \, X_n]$ and a RAC hyperfield homomorphism $f :\H \rightarrow \T$, 
$$ V(f_*(p)) = f(V(p)).$$
\end{theorem}

\begin{proof}
The inclusion $f(V(p)) \subseteq V(f_*(p))$ is a direct consequence of Lemma \ref{pushforwardgeneral}. \\

\noindent The inclusion in the reverse direction, $V(f_*(p)) \subseteq f(V(p))$, is more interesting and requires an argument. Take a point $\underline{a} \in V(f_*(p))$, so $ \0 \in f_*(p)(\underline{a})$. The aim is to demonstrate that there exists an element in $V(p)$ that pushes forward to $\underline{a}$. This will be done by restricting to univariate polynomials and using the property that $f:\H \rightarrow \T$ is a RAC map, to find an appropriate lift of $\underline{a}$.\\ 

\noindent Firstly, choose lifts $\lambda_i \in f^{-1}(a_i)$, where $\underline{a} = (a_1 \, , \dots , \, a_n) \in V(f_*(p))$ and $\underline{\lambda} = (\lambda_1 , \dots , \lambda_n) \in \H^n$. The map $f:\H \rightarrow \T$ can be used to define the coordinate wise map $F : \H^n \rightarrow  \T^n \, , \quad F(x_1, \, \dots \,, x_n):= (f(x_1) \, , \dots \,, f(x_n))$. For any non-zero $D= (d_1 , \dots  , d_n) \in \Z^n$ the map $\varphi(x):= (\lambda_1 \odot x^{d_1} \, , \dots \,, \lambda_n \odot x^{d_n})$ defines an inclusion $\varphi : \H^* \rightarrow  (\H^*)^n$, where $\H^*$ denotes $\H / \{\0\}$. Then, if the map $\psi : \T^* \rightarrow (\T^*)^n$ is defined as
\begin{align}
    \psi(x):&= (f(\lambda_1) \odot x^{d_1} \, , \dots \,, f(\lambda_n) \odot x^{d_n}) \nonumber \\
& = (f(\lambda_1) + d_1x \, , \dots \,, f(\lambda_n) +d_nx ), \nonumber
\end{align}
the diagram below is commutative: 
\begin{center}
\begin{tikzpicture}
\matrix (m) [matrix of math nodes, row sep=6em, column sep=8em, minimum width =4em]
{ \H^* & \T^* \\
(\H^*)^n &  (\T^*)^n \\};
\path[->] (m-1-1) edge [] node [above] {$f$}(m-1-2)
edge [] node [left] {$\varphi$} (m-2-1);
\path[->] (m-2-1) edge [] node [below right] {$F$}(m-2-2);
\path[->] (m-1-2) edge [] node [below right] {$\psi$}(m-2-2);
\end{tikzpicture}
\end{center}
Note that $\psi(0) = F(\underline{\lambda}) =\underline{a}$. The polynomial $p$ will be pulled back through $\varphi$ to a univariate polynomial, then pushed forward through $f_*$ and it will be shown that this polynomial has a root at $0$. The RAC property will be used to lift this root back. The pullback of $p$, denoted $\varphi^{*}p$, is the univariate polynomial defined by the expression for $p$ where $X_i$ is replaced with $\lambda_i \odot X^{d_i}$. Explicitly,
\begin{align}
\varphi^{*}p & = \foo_I c_I \odot (\lambda_1 \odot X^{d_1})^{i_1} \odot \, \dots \, \odot (\lambda_n \odot X^{d_n})^{i_n},\nonumber \\
& = \foo_I c_I \odot \underline{\lambda}^I \odot X^{D \cdot I} \in \H[X]. \nonumber
\end{align}

\noindent The pullback polynomial $\varphi^*p$ is pushed forward to $f_*(\varphi^*p) \in \T[X]$. The image of $f_*(\varphi^*p)$ is equal to the image of $f_*(p)$ when restricted to $\psi$. 
\begin{align}
f_*(p)(\psi(X)) & = f_*(p)\Big(f(\lambda_1) \odot X^{d_1} \, , \dots \, , f(\lambda_n) \odot X^{d_n}\Big) \nonumber \\
& = \foo_I f(c_I) \odot f(\underline{\lambda})^I \odot X^{D\cdot I} \nonumber \\
& = f_*(\varphi^*p)(X). \nonumber
\end{align}

\noindent The next step is to show that $f_*(\varphi^*p)$ has a root at $0 \in \T$. This can be seen as,
$$f_*(\varphi^*p)(0) =  \foo_I f(c_I) \odot f(\underline{\lambda})^I \odot 0^{D\cdot I}.$$
Then, due to the arithmetic over $\T$, 
\begin{align}
    \foo_I f(c_I) \odot f(\underline{\lambda})^I \odot 0^{D\cdot I} 
    & = \foo_I f(c_I) \odot f(\underline{\lambda})^I \nonumber \\
    & = \foo_I f(c_I) \odot \underline{a} ^I \nonumber \\
    & = f_*(p)(\underline{a}) \ni \0.
\end{align} 
This shows that $0 \in V(f_*(\varphi^*p))$. Then, the property that the map $f: \H \rightarrow \T$ is a RAC homomorphism gives that there exists an element $\widetilde{a} \in \H$ such that $f(\widetilde{a}) = 0$ and $\widetilde{a} \in V(\varphi^*p)$.\\

\noindent Furthermore, to ensure that $\tilde{a}$ can be pushed forward to a root of $p$ the tuple $D \in \Z^n$ must be chosen with the following property. Choose $D \in \Z^n$ such that the dot products of $D$ taken with exponent vectors of monomial terms of $p = \foo_{I} c_{I} \odot \underline{X}^{I}$, are all distinct. Explicitly, 
$$ D \cdot J = d_1 \cdot j_1 + \dots + d_n \cdot j_n \neq d_1 \cdot j_1' + \dots + d_n \cdot j_n' = D \cdot J',$$
for all pairs $J,J' \in I$. The condition $D \cdot J \neq D \cdot J'$ can be interpreted as $D$ not lying on the hyperplane defined by $\underline{X} \cdot (J - J')$. Therefore, it is possible to pick such a $D \in \Z^n$, as the number of possible pairs $J,J'$ is finite and $\Z^n$ can not be covered by a finite union of hyperplanes.  \\

\noindent The pullback $\varphi^{*}p$ utilises the requirement imposed on the tuple $D = (d_1 \, , \dots \, , d_n)$. If the requirement was not imposed, this would allow multiple monomials in the restricted polynomial to have equal exponents. This would lead to the corresponding coefficient being a hypersum, thus the potentially detrimental possibility of the restriction becoming a set of polynomials rather than a single polynomial, which is what is needed here. \\

\noindent This condition on $D$ implies that the element $\widetilde{a}$ can then be pushed forward through $\varphi$ to give an element $\varphi(\widetilde{a}) \in V(p)$. As the diagram commutes, this shows that the $F(\varphi(\widetilde{a})) = \psi(f(\widetilde{a})) = \psi(0) = \underline{a}$, which is sufficient to show that for every element of $V(f_*(p))$ there is a lift to an element of $V(p)$. This demonstrates that $V(f_*(p)) \subseteq f(V(p))$, giving the desired result. \\

\noindent Hence, it has been shown that both $f(V(p)) \subseteq V(f_*(p))$ and $V(f_*(p)) \subseteq f(V(p))$ hold. These taken together demonstrate the required equality, $V(f_*(p)) = f(V(p))$.
\end{proof}

\noindent The structure of the proof, restricting to univariate polynomials to use the RAC property is based on the argument presented in a proof of the original theorem in tropical geometry, as seen in \cite{MB}.  Therefore, this adapted proof along with Example \ref{trop} gives a proof which encompasses that of the original Kapranov's Theorem.\\

\noindent In particular, since $\eta:\T\C \rightarrow \T$ is RAC, the above theorem applies to it. 

\section{Characterising RAC Maps.}\label{characterisingRAC}

\noindent This section aims to present sufficient conditions for a hyperfield homomorphism to be a RAC map. This will not include a complete description of the necessary conditions for a hyperfield homomorphism to be a RAC; this remains an open topic. To begin the exploration into RAC hyperfield homomorphisms, a definition linking multiplicities of roots to the degree of the polynomial. 

\begin{definition}
A hyperfield is said to satisfy the \emph{multiplicity bound} if for all univariate polynomials, $p \in \H[X]$,
$$ \sum_{a \in \H} \mathrm{mult}_a(p) \le \mathrm{deg}(p).$$
Furthermore, a hyperfield is said to satisfy the \emph{multiplicity equality} if the above inequality is an equality for all univariate polynomials. 
\end{definition}

\noindent Proposition B from \cite{BL}  describes a relationship between the multiplicities of roots over a field and the multiplicities of the push-forward of these roots over a hyperfield. It is stated below, in notation consistent with this paper. 

\begin{proposition}(\cite[Prop. B]{BL})\label{propB}
Let $K$ be a field and $\H$ a  hyperfield with hyperfield homomorphism $f: K \rightarrow \H$.  Let $p \in K[X]$, with push-forward $f_*(p) \in \H[X]$.Then, 
$$ \mathrm{mult}_b(f_*(p)) \ge \sum_{a \in f^{-1}(b)} \mathrm{mult}_a(p)$$
for all $b \in \H$. Moreover, if $\H$ satisfies the multiplicity bound for the polynomial $f_*(p)$ and $p \in K[X]$ splits into linear factors then, 
$$  \mathrm{mult}_b(f_*(p)) = \sum_{a \in f^{-1}(b)} \mathrm{mult}_a(p).$$
\end{proposition}

\begin{lemma}\label{localbackinclusion}
Let $f: K \rightarrow \H$ be a hyperfield homomorphism, with $\H$ satisfying the multiplicity bound. If the polynomial $p \in K[X]$ splits into linear factors, then $V(f_*(p)) \subseteq f(V(p))$ for $p$.
\end{lemma}
\begin{proof}
Take the polynomial $p \in K[X]$, $f_*(p) \in \H[X]$. Take $b\in V(f_*(p))$, so $ \0 \in f_*(p)(b)$ and hence $\mathrm{mult}_b(f_*(p)) >0$.
As $p$ splits into linear factors over $K$ and $\H$ satisfies the multiplicity bound, by Proposition \ref{propB} we have, 
$$0 < \mathrm{mult}_b(f_*(p)) = \sum_{a \in f^{-1}(b)} \mathrm{mult}_a(p). $$
Therefore, there exists $\hat{a} \in f^{-1}(b)$ such that $\mathrm{mult}_{\hat{a}}(p) >0$. Thus, $\hat{a}$ is a root of $p \in K[X]$. This demonstrates that given an element $b \in V(f_*(p))$ there exists an element $\hat{a}$, such that $f(\hat{a}) = b$ and $p(\hat{a}) = 0$. This shows that $ V(f_*(p)) \subseteq f(V(p))$.
\end{proof}

\noindent The above lemma provides a method for detecting whether a hyperfield homomorphism from a field to a hyperfield is a RAC map. The next stage is to understand whether this view can be extended to maps $f: \H_1 \rightarrow \H_2$. Proposition \ref{propB} is the key tool used in the proof of Lemma \ref{localbackinclusion}. The following discussion aims to explore a generalisation of this result for maps $f: \H_1 \rightarrow \H_2$. Firstly, it is important to recognise that this is not a simple generalisation and the property from Prop. B in \cite{BL} does not always hold in this less restrictive setting.

\begin{example}\label{PKeg}
Take the hyperfield homomorphism $f: \P \rightarrow \K$. Note that, over the Krasner hyperfield $\K$, the multiplicity bound achieves equality for all polynomials: $\sum_{b \in \K}\mathrm{mult}_b(p) = \mathrm{deg}(p)$ for all $p \in \K[X]$, (see Remark 1.11 in \cite{BL} for details). Take the polynomial $p = X^2 \boxplus X \boxplus 1 \in \P[X]$, then due to the argument presented in \cite[remark 1.10]{BL}, $\sum_{a \in \P}\mathrm{mult}_a(p) = \infty$, in particular $a=e^{i\theta}$ is a root of $p$ for all $\frac{\pi}{2} < \theta <\frac{3\pi}{2}$. Now the push-forward coefficients of the polynomial are unchanged, $f_*(p) = X^2 \boxplus X \boxplus 1$, with $\sum_{b \in \K} \mathrm{mult}_b(f(p)) = 2$. This then leads to, 
$$ \infty = \sum_{a \in \P \setminus \{0\}} \mathrm{mult}_a(p) = \sum_{a \in f^{-1}(1)} \mathrm{mult}_a(p) \not \le \mathrm{mult}_1(f(p)) = 2.$$
This demonstrates that the property does not hold in total generality over all hyperfield homomorphisms. 
\end{example}

\noindent There are several key properties of fields that underpin the result in Proposition \ref{propB}. These properties do not automatically hold for hyperfields. The first key property is that all fields satisfy the multiplicity bound. The second has to do with the factorisation process of polynomials. Restricting to hyperfields with the multiplicity bound is a solution to half the problem, whereas the factorisation property needs to be discussed in further detail.
 
If $K$ is a field, then $K[X]$ is a unique factorization domain. Moreover, if $K$ is algebraically closed then every polynomial $p \in K[X]$ factorises completely into a product of linear factors. It is not obvious that these properties extend to hyperfields. This is due to the non-uniqueness of the choice of factorisation, even for more well-behaved hyperfields, such as those with the doubly distributive property (see \cite{BS} for description of doubly distributive hyperfields). It could occur that for two distinct roots, the maximum multiplicity is achieved with different factorisations. In an attempt to overcome this the next definition is introduced.  

\begin{definition}
Given a polynomial $p \in \H[X]$, denote the list of roots, inclusive of repetitions corresponding to multiplicities, by $A_p: =\{a_1 \, \dots \, a_k\}$. A hyperfield $\H$ is said to have the \emph{inheritance} property at $p$, if for all subsets $\{a_{j_1}, \, \dots \, , a_{j_m}\} \subseteq A_p$, with $m \le \mathrm{deg}(p)$, there exists $q \in \H[X]$ such that 
$$ p \in (X-a_{j_1})\odot(X-a_{j_2}) \odot \dots \odot (X-a_{j_m}) \odot q.$$ 
A hyperfield is said to have the \emph{inheritance} property in general if it holds for every polynomial $p \in \H[X]$.
\end{definition}
 
 \noindent The work in this paper does not extend to characterising the multiplicity bound or inheritance properties, but rather opens this area up for exploration. The following conjectures are based on the current knowledge of doubly distributive hyperfields, specifically including $\K,\S$ and $\T$.
 
\begin{conjecture}
All hyperfields with the doubly distributive property satisfies the multiplicity bound.
\end{conjecture}

\begin{conjecture}
All hyperfields with the doubly distributive property satisfies the inheritance property.
 \end{conjecture}
 
\noindent There will now be a demonstration of the implications of the multiplicity bound and the inheritance property.

\begin{lemma}\label{multiboundpfbound}
Given the hyperfield homomorphism $f: \H_1 \rightarrow \H_2$, then for all $p = \foo_{i=0}^n c_i \odot X^i \in \H_1[X]$, 
\begin{equation}\label{PFmultibound}
\mathrm{mult}_b(f_*(p)) \ge \sum_{a \in f^{-1}(b)} \mathrm{mult}_a(p) 
\end{equation}
holds if $\H_1$ satisfies the multiplicity bound and inheritance property. 
\end{lemma}
\begin{proof}
As $\H_1$ satisfies the multiplicity bound the list of roots, including repetitions corresponding to the multiplicities, $\{a_1 \, \dots \, a_k\}$ is a finite set with $k \le \mathrm{deg}(p)$. Take the subset $\{a_{j_1} \, \dots  \, a_{j_m}\} \subseteq \{a_1 \, \dots \, a_k\}$, such that $f(a_{j_1}) = \dots = f(a_{j_m}) = b$. These are the only elements of $\H_1$ that are roots of $p$ and push-forward to $b$. By the inheritance property, there exists a $q \in \H_1[X]$, such that, 
$$ p \in (X - a_{j_1}) \odot \dots \odot (X - a_{j_m}) \odot q.$$
By the hyperfield homomorphism properties it can be seen that under $f_*$, 
\begin{align}
    f_*(p) &\in  (X - f(a_{j_1})) \odot \dots \odot (X - f(a_{j_m})) \odot f_*(q) \nonumber \\
    & \in (X - b) \odot \dots \odot (X - b) \odot f_*(q). \nonumber 
\end{align}
This gives $\mathrm{mult}_b (f_*(p)) \ge m$, implying that, 
$$ \mathrm{mult}_b (f_*(p)) \ge \sum_{a \in f^{-1}(b)} \mathrm{mult}_a (p).$$
\end{proof}

\begin{theorem}\label{boundeqRAC} 
Given $\H_1$ which satisfies the multiplicity equality and has the inheritance property, and $\H_2$ that satisfies the multiplicity bound, then the homomorphism $f: \H_1 \rightarrow \H_2$ is a RAC map. 
\end{theorem}
\begin{proof}
The hyperfield $\H_1$ satisfying the multiplicity equality and the inheritance property implies by Lemma \ref{multiboundpfbound} that, $\mathrm{mult}_b(f_*(p)) \ge \sum_{a \in f^{-1}(b)} \mathrm{mult}_a(p)$ holds. These together then imply that;
$$ \mathrm{deg}(p) = \sum_{a \in \H_1}\mathrm{mult}_a(p) \le \sum_{b \in \H_2}\mathrm{mult}_b(f_*(p)) \le \mathrm{deg}(f_*(p)) = \mathrm{deg}(p).$$
Explicitly, the first equality is given by the multiplicity equality of $\H_1$ and the second inequality holds due to the multiplicity bound on $\H_2$. Yielding, $\mathrm{mult}_b(f_*(p)) = \sum_{a \in f^{-1}(b)} \mathrm{mult}_a(p)$.
Then finally,  $\mathrm{mult}_b(f_*(p)) = \sum_{a \in f^{-1}(b)} \mathrm{mult}_a(p)$ gives that the map $f: \H_1 \rightarrow \H_2$  is RAC, using analogous logic to the proof of Lemma \ref{localbackinclusion}.
\end{proof}

\begin{remark}
Theorem \ref{boundeqRAC} demonstrates that there are sufficient conditions that can be given to both hyperfields to give the corresponding homomorphism as RAC, although this does not classify all RAC maps. It does incorporate the motivating example for the the paper, $\mathit{trop}: K \rightarrow \R \cup\{-\infty\}$. Although, it can be seen in Example \ref{Ex:TC-multi-bound} that $\T\C$ does not satisfy the multiplicity bound, and hence does not fulfil the conditions of Theorem \ref{boundeqRAC}. This demonstrates the theoretical complexity in attempting to outline the conditions for a hyperfield homomorphism to be RAC. 
\end{remark}

\begin{example}\label{Ex:TC-multi-bound}
Given the polynomial $ p = X^2 \boxplus X \boxplus 1 \in \T\C[X]$, then 
\begin{align}
p(-1) & =  (-1)^2 \boxplus -1 \boxplus 1 \nonumber \\
& = 1 \boxplus -1 \boxplus 1 \nonumber \\
& = \Big ( \{ a \in \C \, : \, |a| \le 1 \} \Big) \boxplus 1 \nonumber \\
& \ni -1 \boxplus 1 \ni 0 \nonumber 
\end{align}
\begin{align}
p(i) & = -1 \boxplus i \boxplus 1 \nonumber \\
& = \Big ( \{\text{shortest closed arc between} \, -1 \, \text{and} \, i \} \Big) \boxplus 1 \nonumber \\
& \ni -1 \boxplus 1 \ni 0 \nonumber 
\end{align}
\begin{align}
p(-i) & = -1 \boxplus -i \boxplus 1 \nonumber \\
& = \Big ( \{\text{shortest closed arc between} \, -1 \, \text{and} \, -i \} \Big) \boxplus 1 \nonumber \\
& \ni -1 \boxplus 1 \ni 0 \nonumber 
\end{align}

\noindent These three calculations over $\T\C$ imply that $\{-1,i,-i \} \subset V(p) \subset \T\C$. This shows that a degree 2 polynomial can have three distinct root over $\T\C$, thus the multiplicity bound does not hold. 
\end{example}

\bibliographystyle{plain}
\bibliography{References}

\end{document}